\documentclass[12pt]{article}
\usepackage{amsmath,amsfonts,amssymb,a4,epsfig,array,psfrag}
\usepackage{amsthm,amsfonts}
\usepackage{enumerate}
\usepackage{xcolor}
\usepackage{etoolbox} 

\newcommand\llangle{\langle\!\langle}
\newcommand\rrangle{\rangle\!\rangle}

\newtheorem{theo}{Theorem}
\newtheorem{lemma}{Lemma}[section]

\newtheorem{prop}[lemma]{Proposition}

\newtheorem{coro}[lemma]{Corollary}

\theoremstyle{remark}

\AtEndEnvironment{defin}{\null\hfill$\diamond$}%

\AtEndEnvironment{example}{\null\hfill$\diamond$}%
\newtheorem{remark}[lemma]{Remark}
\AtEndEnvironment{remark}{\null\hfill$\diamond$}%

\newcommand{\real}{\mathbb{R}}

\newcommand{\cP}{\mathcal{P}}
\newcommand{\cC}{\mathcal{C}}

\newcommand{\cS}{S}

\newcommand{\diag}{\operatorname{diag}}
\newcommand{\penta}{\operatorname{penta}}
\newcommand{\tr}{\operatorname{tr}}

\newcommand{\hull}{{\operatorname{conv}}}
\newcommand{\RR}{\mathbb{R}}
\newcommand{\Ss}{\mathbb{S}}

\newcommand{\bin}{\mathbf{in}}
\newcommand{\bout}{\mathbf{out}}

\newcommand{\Ein}{E_{\bin}}

\newcommand{\Eout}{E_{\bout}}
\newcommand{\hEout}{{\hat E}_{\bout}}

\newcommand{\Ptinout}{\cP_t(\Ein,\Eout)}

\newcommand{\Ctinout}{\cC_t(\Ein,\Eout)}

\newcommand{\Sinout}{{\mathcal S}_t(\Ein,\Eout)}

\newcommand{\Sfinout}{{\mathcal S}_f(\Ein,\Eout)}

\newcommand{\Stinout}{{\mathcal S}_t(\Ein,\Eout)}

\title{Higher dimensional versions of \\ theorems of Euler and Fuss}
\author{Peter Gibson \and Nicolau Saldanha \and Carlos Tomei}

\makeindex

\begin{document}
\maketitle

\begin{abstract}
We present higher dimensional versions of the classical results
of Euler and Fuss, both of which are special cases
of the celebrated Poncelet porism.
Our results concern polytopes,
specifically simplices, parallelotopes and cross polytopes,
inscribed in a given ellipsoid
and circumscribed to another.
The statements and proofs use the language of linear algebra.
Without loss, one of the ellipsoids is the unit sphere
and the other one is also centered at the origin.
Let $A$ be the positive symmetric matrix 
taking the outer ellipsoid to the inner one.
If $\tr A = 1$, there exists a bijection between
the orthogonal group $O(n)$ and the set of such labeled simplices.
Similarly, if $\tr A^2 = 1$, there are families of parallelotopes
and of cross polytopes, also indexed by $O(n)$.
\end{abstract}



\section{Introduction}

We recall the classic geometric Poncelet porism
\cite{AZ,Flatto,Tabashnikov}.
Consider two disjoint ellipses $\Eout, \Ein \subset \RR^2$
with $\Ein$ contained in $\hull(\Eout) \subset \RR^2$,
the convex hull of $\Eout$.
A polygon with $k$ vertices $P_0, P_1, P_2, \ldots P_k = P_0$
\textit{fits tightly} between $\Eout$ and $\Ein$
if its vertices belong to $\Eout$
and its sides are tangent to $\Ein$.
We implicitly assume $P_{i+2} \ne P_i$ (for all $i$).
Alternatively, we say the polygon is \textit{tight}.
Figure \ref{fig:euler-ellipse} shows two pairs of ellipses:
in the first example several triangles ($k = 3$) fit tightly;
in the second, several quadrilaterals ($k = 4$) also fit tightly.

\begin{figure}[ht]
\def\svgwidth{130mm}
\centerline{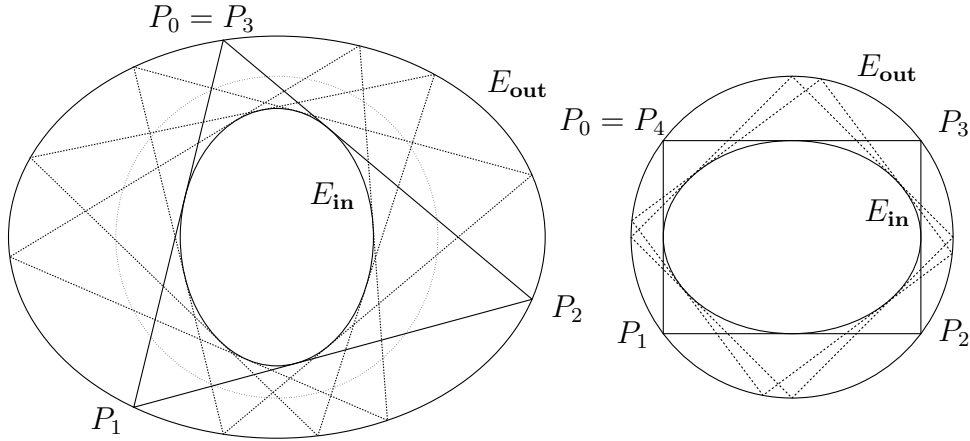}
\caption{Poncelet porism, $k = 3$ and $k = 4$}
\label{fig:euler-ellipse}
\end{figure}

\begin{theo}[Poncelet porism]
\label{theo:poncelet}
If the pair $\Eout, \Ein$ admits a tight polygon with $k$ vertices
then any point $Q_0 \in \Eout$ is a vertex of a tight polygon with $k$ vertices.
\end{theo}

A little projective geometry \cite{Coxeter}
shows that the general case above
follows from the special case when $\Eout$ and $\Ein$ are both circles.
The special cases $k = 3$ and $k = 4$ of Theorem \ref{theo:poncelet}
(for circles)
were proved earlier by Euler and his student Fuss, respectively
\cite{Gheorghe}.
If $\Eout$ and $\Ein$ are circles of radii $R$ and $r$, respectively,
and the distance between the two centers is $d$,
then there exist tight triangles ($k = 3$) or
quadrilaterals ($k = 4$) if and only if
\begin{equation}
\label{equation:radii}
\frac{1}{(R-d)^{k-2}} + \frac{1}{(R+d)^{k-2}} = \frac{1}{r^{k-2}}, 
\qquad k \in \{3,4\}. 
\end{equation}


Projective geometry also implies that the general case
follows from another special case,
with $\Eout$ and $\Ein$ being ellipses centered at the origin,
as in Figure~\ref{fig:euler-ellipse}.
We state higher dimensional versions of the theorems of Euler and Fuss
in this context, starting with Fuss.

Consider two disjoint ellipsoids $\Eout, \Ein \subset \RR^n$
centered at the origin with $\Ein \subset \hull(\Eout)$.
By applying a linear transformation,
we assume that $\Eout = \Ss^{n-1}$, the unit sphere.
Clearly, there exists a unique positive symmetric matrix $A$
with $A\Eout = \Ein$.

A closed, convex polytope $P \subset \RR^n$
\textit{fits} between $\Eout$ and $\Ein$
if $\Ein \subset P \subset \hull(\Eout)$.
The polytope $P$ is \textit{inscribed} in $\Eout$
if all its vertices belong to $\Eout$
and $P$ is \textit{circumscribed} to $\Ein$
if all its hyperfaces are tangent to $\Ein$.
It fits \textit{tightly} between $\Eout$ and $\Ein$
if it is inscribed in $\Eout$ and circumscribed to $\Ein$.
We  usually consider \textit{labeled} polytopes, for which the
vertices are indexed.

A \textit{centrally symmetric parallelotope} in $\RR^n$
is a convex polytope with $2^n$ vertices of the form
$\pm v_1 \pm v_2 \pm \cdots \pm v_n$ where the vectors
$v_1, \ldots, v_n$ form a basis.
Thus, for $n = 2$ the polytope is a parallelogram
and for $n=3$, a parallelepiped.
A \textit{label} for a parallelotope is
a family $(v_k)_{1 \le k \le n}$ of vectors as above
so that each vertex in turn is labeled by a sequence of signs.
A centrally symmetric parallelotope is \textit{orthogonal}
if the basis $(v_k)$ is orthogonal.
As we shall see, if a parallelotope is inscribed in the unit sphere
then it is necessarily centrally symmetric and orthogonal.
For $\Eout$ and $\Ein$ as above,
let $\Ptinout$ be the set of all labeled
parallelotopes fitting tightly between $\Eout$ and $\Ein$.

As usual, $O(n)$ is the real orthogonal group.
Define the map
\[ \phi: \Ptinout \to O(n) \]
taking a parallelotope with label $(v_k)$
to the matrix $Q \in O(n)$ whose columns are obtained
from the basis $(v_k)$ by Gram-Schmidt orthonormalization.

\begin{theo}
\label{theo:fuss}
Let $\Eout = \Ss^{n-1}$ and $\Ein = A\Eout$, as above.
The set $\Ptinout$ is nonempty if and only if $\tr(A^2) = 1$.
In this case, the map $\phi: \Ptinout \to O(n)$
is a diffeomorphism.
\end{theo}

The \textit{dual} or \textit{polar} of a bounded convex set
$X \subset \RR^n$ with $0 \in X$ is 
\[ \tilde X = \{ v \in \RR^n \;|\;
\forall w \in X, \langle v,w \rangle \ge -1 \}. \]
For instance, the dual of a centrally symmetric parallelotope
is a centrally symmetric cross polytope
(see Remark~\ref{remark:crosspolytope}).
We shall discuss duality further in Section \ref{section:construct}.

For $\Eout$ and $\Ein$ as above,
let $\Ctinout$ be the set of all labeled centrally symmetric
cross polytopes fitting tightly between $\Eout$ and $\Ein$.
The map $\phi: \Ctinout \to O(n)$ is defined similarly
to the parallelotope situation (see Section \ref{section:fuss}).
Duality applied to Theorem \ref{theo:fuss} above
gives us the following similar result:

\begin{coro}
\label{coro:fuss}
Let $\Ein = \Ss^{n-1}$ and $\Eout = A^{-1}\Eout$.
The set $\Ctinout$ is nonempty if and only if $\tr(A^2) = 1$.
Then, the map $\phi: \Ctinout \to O(n)$
is a diffeomorphism.
\end{coro}

\bigskip

We now extend Euler's theorem to \textit{simplices} in $\RR^n$:
convex polytopes with $n+1$ vertices and nonempty interior.
For ellipsoids $\Eout$ and $\Ein$ as above,
let $\Sfinout$ (resp. $\Sinout$) be the set of labeled simplices
with vertices $v_1, \ldots, v_n, v_{n+1} \in \hull(\Eout)$
fitting (resp. fitting tightly) between $\Eout$ and $\Ein$,
so that
$\Sinout \subseteq \Sfinout$.



Define the map
\[ \phi: \Sfinout \to O(n) \]
taking a simplex $S \in \Sfinout$ with vertices $v_1, \ldots, v_n, v_{n+1}$
to the matrix $Q \in O(n)$ whose columns are obtained
from the basis $v_1, \ldots, v_n$ by Gram-Schmidt orthonormalization.

\bigskip


\begin{theo}
\label{theo:euler}
Let $\Ein = \Ss^{n-1} \subset \RR^n$ be the unit sphere.
Let $\Eout \subset \RR^n$ be a nondegenerate ellipsoid
centered at the origin.
Let $A$ be the unique positive definite real symmetric matrix $A$
such that $A\Eout = \Ein$.
\begin{enumerate}[(i)]
\item{If $\tr(A) > 1$,
no simplex fits between
$\Ein$ and $\Eout$:
$\Sfinout = \varnothing$.}
\item{If $\tr(A) = 1$,
every fitting simplex fits tightly: $\Sfinout = \Stinout$.
The map $\phi: \Sfinout \to O(n)$ is a diffeomorphism.}
\item{If $\tr(A) < 1$, $\Stinout \ne \Sfinout$.
The map $\phi: \Sfinout \to O(n)$ is surjective and not injective.}
\end{enumerate}
\end{theo}

\begin{remark}
\label{remark:supereuler}
In the situation of item (iii) above,
the cases $n = 2$ and $n \ge 3$ behave differently.
For $n = 2$, $\Stinout = \varnothing$:
this follows from the case $k=3$ of Poncelet porism.
If $n \ge 3$,
the restriction $\phi: \Stinout \to O(n)$
is still surjective and not injective.  
We shall present illustrative examples
but not the cumbersome computations.
\end{remark}

\bigskip

Our extensions of the classical results of Euler and Fuss
to higher dimensions provide a simple counterpart to
Equation~\eqref{equation:radii} in terms of traces.
This was possible because of our choice to work
with centrally symmetric ellipsoids.






\bigskip

Research by the first author
was partially supported by NSERC grant DG RGPIN-2022-04547.
Research by the second and third author 
was partially supported by CNPq, Faperj and Capes (Brazil).

\section{Fuss and parallelotopes}
\label{section:fuss}

\begin{lemma}
\label{lemma:ortho}
If a centrally symmetric parallelotope is inscribed in a sphere
then it is orthogonal.
\end{lemma}

\begin{proof}
Notice first that for $n = 2$ the lemma says
that a parallelogram inscribed in a circle is a rectangle:
this follows from the fact that both diagonals have the same length.

For the general case,
consider a centrally symmetric parallelotope 
with vertices $\pm w_1 \pm \cdots \pm w_n$
inscribed in a sphere $\Eout$ of radius $R$.
Consider $i$ and $j$ with $1 \le i < j \le n$:
the four vertices 
$\tilde w \pm w_i \pm w_j$, 
$\tilde w = - w_i - w_j + \sum_k w_k$,
are inscribed in the intersection of a $2$-dimensional plane
with the sphere, which is of course a circle.
The case $n = 2$ implies $w_i \perp w_j$.
\end{proof}

\begin{remark}
For an alternative proof, use the notation above and notice that
$|\tilde w \pm w_i \pm w_j| = R$ for any choice of signs.
We then have
\[
\langle \tilde w \pm w_j, w_i \rangle = \frac14
\left( |\tilde w \pm w_j + w_i|^2 - |\tilde w \pm w_j - w_i|^2 \right) =
\frac{R^2 - R^2}{4} = 0 
\]
and therefore $\langle w_i, w_j \rangle = 0$.
\end{remark}

\begin{proof}[Proof of Theorem \ref{theo:fuss}]
We first prove that if $\Ptinout \ne \varnothing$ then $\tr(A^2) = 1$.
Take $P \in \Ptinout$.
From Lemma \ref{lemma:ortho}, $P$ is an orthogonal parallelotope.
The edges of $P$ are parallel to $q_k$,
where $(q_k)$ is an orthonormal basis.
The map $\phi$ takes $P$ to $Q \in O(n)$ with columns $(q_k)$.

Let $2\ell_k$ be the length of the edge in the direction $q_k$
so that the vertices of $P$ are
$\pm \ell_1 q_1 \pm \cdots \pm \ell_k q_k \pm \cdots \pm \ell_n q_n$;
the faces are the hyperplanes $H_{k,\pm}$ of equations
$\langle \cdot, q_k \rangle = \pm \ell_k$.
By Pythagoras we have $\sum_k \ell_k^2 = 1$.
Let $v_k$ be the point of tangency between $\Ein$ and $H_{k,+}$.
By definition of $A$, $v_k$ has the form $v_k = Au_k$, $u_k \in \Eout$.
We determine $u_k$ and $v_k$.

The vector $u_k$ maximizes $\langle q_k, Ax \rangle$
with the restriction $|x| = 1$.
Since $A$ is symmetric, 
$\langle q_k, Ax \rangle = \langle x, Aq_k \rangle$
and then
\[ u_k = x = \frac{A q_k}{|A q_k|}, \qquad
v_k = A u_k = \frac{A^2 q_k}{|A q_k|}. \]
Since $\langle v_k, q_k \rangle = \ell_k$, we have
\[ \ell_k =
\frac{\langle A^2 q_k, q_k \rangle}{|A q_k|} = 
\frac{\langle A q_k, A q_k \rangle}{|A q_k|} = 
|A q_k|. \]
We compute the trace in the basis $(q_k)$, 
thus proving the first claim:
\[ \tr(A^2) = \sum_k \langle A^2 q_k, q_k \rangle =
\sum_k \langle A q_k, A q_k \rangle = 
\sum_k |A q_k|^2 = \sum_k \ell_k^2 = 1. \]

We are left with proving that
the map $\phi: \Ptinout \to O(n)$ constructed above is invertible.
Take $Q \in O(n)$ with columns $(q_k)$.
Consider the parallelotope $P$ with vertices
$\pm \ell_1 q_1 \pm \cdots \pm \ell_k q_k \pm \cdots \pm \ell_n q_n$
where each $\ell_k > 0$ is chosen so that the ellipsoid $\Ein$
is tangent to the hyperplane $H_{k,+}$ of equation
$\langle \cdot, q_k \rangle = \ell_k$.
We prove that $P \in \Ptinout$ and $\phi(P) = Q$.
Clearly, $P$ is inscribed in a sphere $R \Eout$
of radius $R > 0$, $R^2 = \sum \ell_k^2$.
We need to verify that $R = 1$.

Indeed, consider a scaled parallelotope $\hat P = (1/R) P$,
inscribed in $\Eout$
and circumscribed to $\widehat\Ein = (1/R) \Ein$.
The matrix $\hat A = (1/R) A$ satisfies $\hat A \Eout = \widehat\Ein$
and therefore, from the previous paragraph, $\tr(\hat A^2) = 1$.
But $\tr(\hat A^2) = R^{-2}$, so that $R = 1$, as desired.
The fact that $\phi(P) = Q$ is obvious.
\end{proof}

\begin{remark}
The hypothesis of polytopes being
centrally symmetric parallelotopes is essential.
Indeed, for dimension $n = 3$ consider parameters $r > 0$ and $s \ge 1$
and construct the convex polyhedron $P$
with $8$ vertices:
$(\pm rs, \pm rs^{-1}, r)$, $(\pm rs^{-1}, \pm rs, -r)$.
For $s = 1$, $P$ is a cube.
In general, this polyhedron has $6$ faces:
two rectangles in the planes $z = \pm r$ and four trapezoids,
as shown in Figure \ref{fig:fakecube}.

\begin{figure}[ht]
\def\svgwidth{75mm}
\centerline{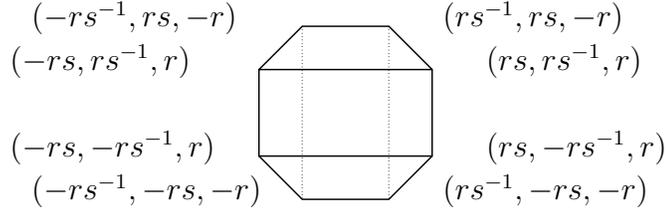}
\caption{A convex polytope which is both inscribable and circumscribable.}
\label{fig:fakecube}
\end{figure}

A simple computation verifies that $P$ is inscribed
in the sphere $\Eout$ of radius $R = r \sqrt{s^2 + 1 + s^{-2}}$;
$P$ is also circumscribed to the sphere $\Ein$ of radius $r$.
The positive symmetric matrix $A$ with $A\Eout = \Ein$
is therefore $A = (r/R) I$, with trace $\tr(A) = 3/\sqrt{s^2 + 1 + s^{-2}}$.
By adjusting the value of $s$,
$\tr(A)$ can assume any value in the interval $(0,1)$.
Similar constructions are possible in higher dimension ($n > 3$)
but not in the plane ($n=2$).
The situation is reminiscent of Remark~\ref{remark:supereuler}.
\end{remark}

\begin{remark}
\label{remark:crosspolytope}
We provide details concerning Corollary~\ref{coro:fuss}.
A \textit{centrally symmetric cross polytope} in $\RR^n$
is a convex polytope with non empty interior and
vertices $\pm v_1, \ldots, \pm v_n$.
For $n = 2$, a cross polytope is a parallelogram;
for $n = 3$, an octahedron.


It is easy to verify that
the dual of a centrally symmetric parallelotope
is a centrally symmetric cross polytope.
Thus, Corollary~\ref{coro:fuss}
follows from Theorem \ref{theo:fuss} via duality, and vice versa.
Define the map $\phi: \Ctinout \to O(n)$
taking a cross polytope with vertices $\pm v_k$
to $Q \in O(n)$ with $Qe_k = v_k/|v_k|$:
the map $\phi$ is a diffeomorphism.
\end{remark}


\section{Euler and simplices}
\label{section:simplices}

As in the statement of Theorem \ref{theo:euler},
$A$ is the unique positive definite real symmetric matrix
such that $A(\Eout) = \Ein = \Ss^{n-1}$.
We then have
\[ \Ein = \{ v \in \RR^n \;|\; \langle v,v\rangle = 1 \}, \qquad
\Eout = \{ v \in \RR^n \;|\; \langle Av,Av\rangle = 1 \}. \]
Bases are understood to be families such as $(v_i)_{1 \le i \le n}$:
labeling is important.

\begin{prop}
\label{prop:tracebound}
If $\Sfinout \ne \varnothing$  then $\tr A \le 1$.
Moreover, if $\tr A = 1$, then $\Sfinout = \Stinout$.
\end{prop}

Thus, there are no fitting simplices if $\tr A > 1$.
For $\tr A = 1$, a fitting simplex is tight.
Part (i) and the first claim in part (ii) of Theorem \ref{theo:euler}
follow from the proposition.


\begin{proof}[Proof of Proposition \ref{prop:tracebound}]
Take $\cS \in \Sfinout$ with vertices $(v_i)_{1 \le i \le n+1}$.
In particular,
$\hull(\{v_1,\ldots,v_{n+1}\})$
contains the origin in its interior.
A hyperface $F_i \subset \cS$ is the convex closure
of the vertices $v_j$, $j \ne i$,
and belongs to a hyperplane $H_i$.
Take $w_i \in \Ss^{n-1}$ the closest point to $H_i$
and $t_i \ge 1$ such that $t_iw_i \in H_i$:
\[ H_i =
\{ q \in \RR^n \ |
\ \langle q, t_i w_i \rangle = \langle t_i w_i, t_i w_i \rangle= t_i^2\}. \]
We must then have $\langle v_i , w_j \rangle = t_i \ge 1$ for $i \ne j$;
on the other hand, $\langle v_i, w_i \rangle < 0$,
otherwise  $0$ would not belong to the interior of
$\hull(\{v_1,\ldots,v_{n+1}\})$.

Define extended vectors
$\hat v_i = (v_i, 1), \hat w_j = (w_j, -t_j) \in \RR^{n+1}$.
For $i \ne j$ we have $\langle \hat v_i, \hat w_j \rangle = 0$;
also, $\langle \hat v_i, \hat w_i \rangle < - t_i \le -1$.
We show that the families
$(\hat v_i)_{1 \le i \le n+1}$ and
$(\hat w_i)_{1 \le i \le n+1}$
form bases of $\RR^{n+1}$.
Indeed, if $\sum_i c_i \hat v_i = 0$,
the inner product with $\hat w_j$ gives
$c_i \langle \hat v_i, \hat w_i \rangle =0$, so that $c_i = 0$.
The same argument applies to $(\hat w_i)$.
Said differently, $(\hat v_i)$ and $(\hat w_i)$ are biorthogonal bases.

We phrase these properties in matrix notation.
Let ${\hat V}$ and ${\hat W}$
be $(n+1) \times (n+1)$ matrices
with columns given by $(\hat v_i)$ and $(\hat w_i)$ respectively.
We then have
\[ {\hat W}^T {\hat V} = {D} =
\diag(\langle \hat v_1, \hat w_1 \rangle, \ldots,
\langle \hat v_{n+1}, \hat w_{n+1} \rangle)
= \diag(-t_1,\ldots,-t_{n+1}). \]
We have ${D} \le - I$
in the sense that $\langle u, (D+I) u \rangle \le 0$ for all $u \in \RR^{n+1}$.
The trace of a real $(n+1) \times (n+1)$ matrix $X$ is
\begin{equation}
\label{equation:trace}
\tr X =
\tr {\hat W}^T X ({\hat W}^T)^{-1} =
\tr {\hat W}^T X {\hat V}{D^{-1}}
=
\sum_i \frac{ \langle\hat w_i, X \hat v_i \rangle}{(-t_i)}.
\end{equation}

Let $A_{-} = A \oplus (-1)$ be the $(n+1)\times(n+1)$ matrix
obtained from $A$ by adding a final row and column of zeroes
and an entry equal to $-1$ in position $(n+1,n+1)$.
From Equation \eqref{equation:trace} for $X = A_{-}$,
\begin{equation} \label{equation:traceaminus}
 \tr A_{-} =
\sum_i \frac{ \langle\hat w_i, A_{-} \hat v_i \rangle}{(-t_i)}
= \sum_i \frac{ \langle w_i,  A  v_i \rangle + t_i}{(-t_i)}.
\end{equation}
As $w_i \in \Ss^{n-1}$, $t_i \ge 1$ and $A v_i \in \hull(\Ein)$,
the numerators are greater or equal to zero.
Thus $\tr A_{-} = \tr A - 1 \le 0$.

Now suppose $\tr A = 1$ and $S \in \Sfinout$.
From the computations above,
$\tr A = 1$ if and only if the numerators in
Equation (\ref{equation:traceaminus}) are equal to zero,
that is,
if and only if $\langle w_i, A v_i \rangle = -t_i$ for all $i$.
As $w_i \in \Ss^{n-1}$, and $A v_i \in \hull(\Ein)$,
we have (by Cauchy-Schwartz)
$\langle w_i, A v_i \rangle \ge -1$
with equality if and only if $A v_i = - w_i$.
Since $t_i \ge 1$
we must have $t_i = 1$ and $A v_i = - w_i$.
Thus, the hyperplane $H_i$ containing the face $F_i$ is tangent to $\Ein$:
$S$ is circumscribed to $\Ein$.
Moreover, $w_i = -A v_i$ implies $v_i \in \Eout$, i.e.,
$S$ is inscribed in $\Eout$,
proving that $S \in \Stinout$, and therefore that $\Stinout = \Sfinout$.
\end{proof}

\begin{remark}
\label{remark:tighttetrahedron}
Take $n = 3$ and $a, b, c > 1$. Let
\[ \Eout = \left\{ (x,y,z) \in \RR^3 \;|\;
\frac{x^2}{a^2}+\frac{y^2}{b^2}+\frac{z^2}{c^2} = 1 \right\}, \qquad
A = \begin{pmatrix}
\frac{1}{a} & 0 & 0 \\ 0 & \frac{1}{b} & 0 \\ 0 & 0 & \frac{1}{c}
\end{pmatrix}. \]
For $\tilde w_1 = -(x,y,z) \in \Ein = \Ss^2$,
consider the tetrahedron $S$ with vertices
$v_1 = -A^{-1}\tilde w_1 = (ax,by,cz) \in \Eout$, $v_2 = (ax,-by,-cz)$, 
$v_3 = (-ax,by,-cz)$ and $v_4 = (-ax,-by,cz)$,
so that $S$ is inscribed in $\Eout$.
These four vertices are alternating vertices of a rectangular
parallelepiped with edges parallel to the axes.
Following the construction in the proof of
Proposition \ref{prop:tracebound}, and setting
\[
F(x,y,z) = (ax)^{-2} + (by)^{-2} + (cz)^{-2},
\]
 we have
$t_1 = t_2 = t_3 = t_4$ and
\[
w_1 = - \frac{1}{t_1F(x,y,z)}
\left(\frac{1}{ax}, \frac{1}{by}, \frac{1}{cz} \right), \quad
F(x,y,z) = t_1^{-2}. \]
Thus, $S$ is circumscribed to a sphere of radius $t_1$:
it fits (resp. tightly) between $\Ein$ and $\Eout$
if and only if $F(x,y,z) \le 1$ (resp. $F(x,y,z) = 1$).


In order to study the function $F$,
it suffices to consider the octant $x, y, z \ge 0$:
$F$ goes to infinity at the boundary and has a unique critical point,
a global minimum, at
$x^2 = a^{-1}/\tr A$, $y^2 = b^{-1}/\tr A$, $z^2 = c^{-1}/\tr A$.
Thus, the minimum of $F$ is $(\tr A)^2$.
If $\tr A > 1$, no tetrahedron in this family fits,
consistently with Proposition \ref{prop:tracebound}.
If $\tr A = 1$ we construct
a unique tetrahedron which fits, and fits tightly.
If $\tr A < 1$, there exists a closed disk of values
of $(x,y,z)$ in the first octant for which the tetrahedron fits;
on the boundary of the disk, the tetrahedron fits tightly,
consistently with Remark \ref{remark:supereuler}.

Given $\Eout$ with $\tr A < 1$,
there exist similar tight tetrahedra in other positions;
the algebra for such examples is far more complicated.
\end{remark}

\bigskip

\section{Constructing simplices}
\label{section:construct}

In this section we complete the proof of Theorem \ref{theo:euler}.
More concretely, 
we construct the map $\phi^{-1}: O(n) \to \Sfinout$ when $\tr(A) = 1$.

For a tight simplex $S$, let 
$(v_i)_{1 \le i \le n+1}$
denote its vertices and 
$(w_i)_{1 \le i \le n+1}$
the family of points of tangency of its hyperfaces with $\Ein$.
Set $\tilde v_i =  - A^{-1} w_i$ and $\tilde w_i = - A v_i$.
The simplex $\tilde S$ with vertices $(\tilde v_i)$
is another tight simplex for the same pair of ellipsoids,
as follows from expanding the corresponding algebraic expressions.
The points of tangency to the hyperfaces of $\tilde S$ are $(\tilde w_i)$.
We call $\tilde S$ the \textit{dual} of $S$;
$S$ is \textit{self-dual} if $\tilde S = S$.
Figure \ref{fig:selfdual} shows an example with $n = 2$:
a self-dual triangle $S \in \Stinout$.
We have $\langle v_i , w_j \rangle = 1$ for $i \ne j$,
which in turn gives that $v_i$ is equidistant to all points $w_j$,
since $w_j \in \Ein = \Ss^{n-1}$, a simple geometric fact.

\begin{figure}[ht]
\def\svgwidth{90mm}
\centerline{
\begingroup%
  \makeatletter%
  \providecommand\color[2][]{%
    \errmessage{(Inkscape) Color is used for the text in Inkscape, but the package 'color.sty' is not loaded}%
    \renewcommand\color[2][]{}%
  }%
  \providecommand\transparent[1]{%
    \errmessage{(Inkscape) Transparency is used (non-zero) for the text in Inkscape, but the package 'transparent.sty' is not loaded}%
    \renewcommand\transparent[1]{}%
  }%
  \providecommand\rotatebox[2]{#2}%
  \newcommand*\fsize{\dimexpr\f@size pt\relax}%
  \newcommand*\lineheight[1]{\fontsize{\fsize}{#1\fsize}\selectfont}%
  \ifx\svgwidth\undefined%
    \setlength{\unitlength}{947.39803299bp}%
    \ifx\svgscale\undefined%
      \relax%
    \else%
      \setlength{\unitlength}{\unitlength * \real{\svgscale}}%
    \fi%
  \else%
    \setlength{\unitlength}{\svgwidth}%
  \fi%
  \global\let\svgwidth\undefined%
  \global\let\svgscale\undefined%
  \makeatother%
  \begin{picture}(1,0.69946263)%
    \lineheight{1}%
    \setlength\tabcolsep{0pt}%
    \put(0,0){\includegraphics[width=\unitlength,page=1]{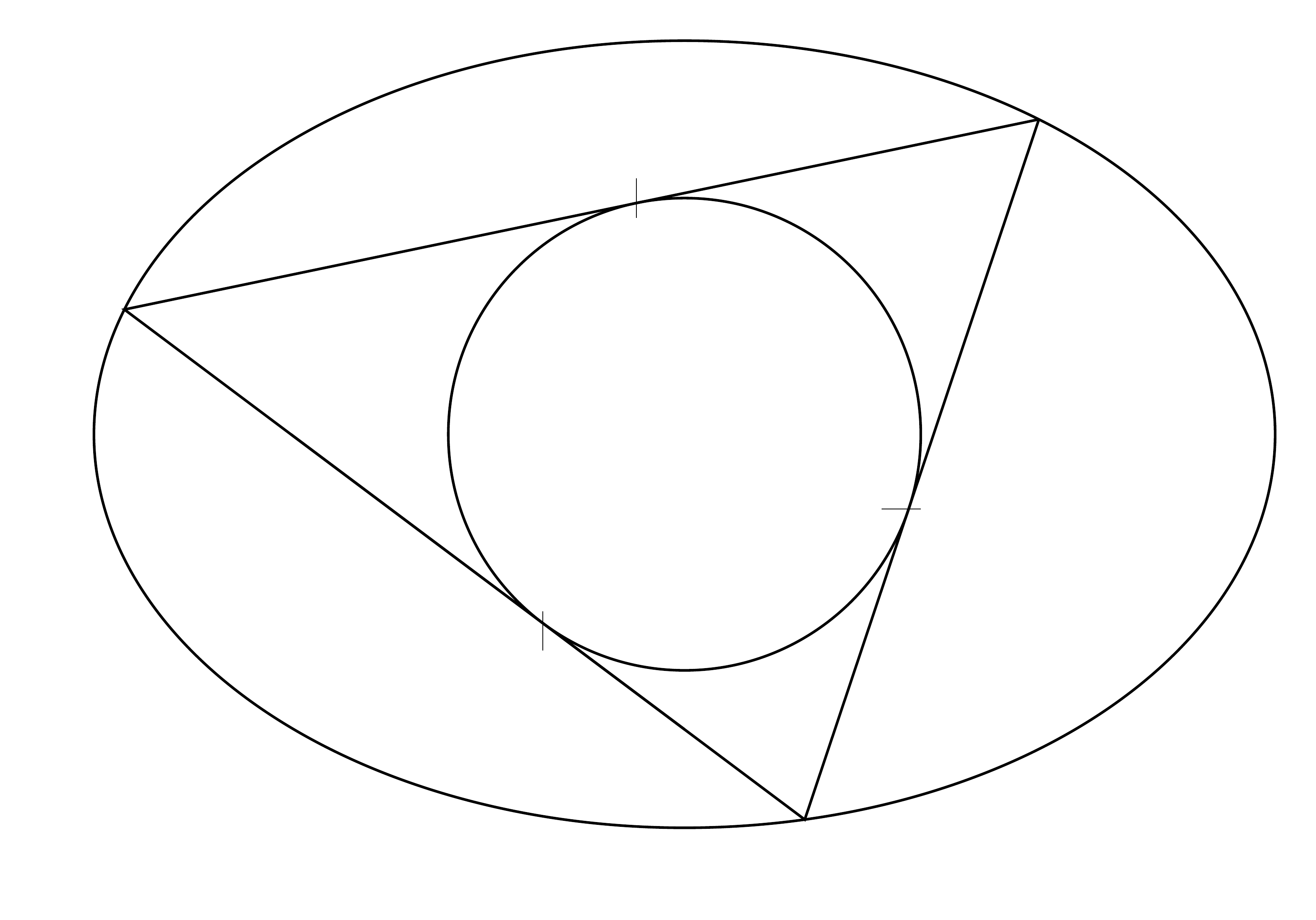}}%
    \put(0.80151326,0.62366032){\color[rgb]{0,0,0}\makebox(0,0)[lt]{\lineheight{1.25}\smash{\begin{tabular}[t]{l}$v_1$\end{tabular}}}}%
    \put(0.02949901,0.46806055){\color[rgb]{0,0,0}\makebox(0,0)[lt]{\lineheight{1.25}\smash{\begin{tabular}[t]{l}$v_2$\end{tabular}}}}%
    \put(0.35566007,0.57877578){\color[rgb]{0,0,0}\makebox(0,0)[lt]{\lineheight{1.25}\smash{\begin{tabular}[t]{l}$w_3 = -Av_3$\end{tabular}}}}%
    \put(0.72072107,0.30647618){\color[rgb]{0,0,0}\makebox(0,0)[lt]{\lineheight{1.25}\smash{\begin{tabular}[t]{l}$w_2 = -Av_2$\end{tabular}}}}%
    \put(0.62197506,0.03118427){\color[rgb]{0,0,0}\makebox(0,0)[lt]{\lineheight{1.25}\smash{\begin{tabular}[t]{l}$v_3$\end{tabular}}}}%
    \put(0.22100642,0.17780714){\color[rgb]{0,0,0}\makebox(0,0)[lt]{\lineheight{1.25}\smash{\begin{tabular}[t]{l}$w_1 = -Av_1$\end{tabular}}}}%
  \end{picture}%
\endgroup%
}
\caption{A self-dual triangle}
\label{fig:selfdual}
\end{figure}
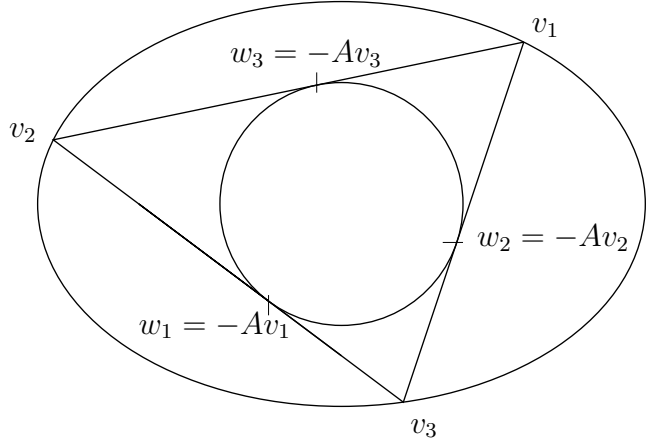

\begin{remark}
\label{remark:thintetrahedron}
Recall that in Remark  \ref{remark:tighttetrahedron},
$S$ has vertices $(\pm ax,\pm by,\pm cz)$
with an even number of negative signs, where $(x,y,z) \in \Ein = \Ss^2$.
In the notation of Proposition~\ref{prop:tracebound},
the tetrahedron $S$ is tight if and only if $t_1 = 1$.
In this case, the dual $\tilde S$ has vertices
\[ \tilde v_i =
\left( \pm \frac{1}{x}, \pm \frac{1}{y}, \pm \frac{1}{z} \right), \]
again with an even number of negative signs.
Thus, $S$ is self-dual precisely when $\tr A = 1$,
consistently with Corollary~\ref{coro:selfdual} below.
\end{remark}

\begin{coro}
\label{coro:selfdual}
If $\tr A = 1$ then
any simplex $S$ which fits between $\Ein$ and $\Eout$ is self-dual.
\end{coro}

\begin{proof}
This is the content of the last paragraph of the proof of
Proposition \ref{prop:tracebound}:
if $\tr A = 1$ then $w_i = -A v_i = \tilde w_i$ (for all $i$).
\end{proof}

\medskip

Given a tight simplex $S \in \Stinout$, its  vertices $v_i$
must satisfy $\| A v_i \| = 1$.
From Corollary \ref{coro:selfdual},
the hyperfaces of $S$ are tangent to $\Ein$ at $w_i = - A v_i$,
so that $\langle - A v_i , v_j \rangle = 1$, for $i \ne j$.
We embed $\RR^n$ into $\RR^{n+1}$ with an extra final coordinate:
given $v \in \RR^n$, we \textit{lift} it to obtain $\hat v = (v,1)$.
We denote the lifted hyperplane by $\hat \RR^n \subset \RR^{n+1}$.
In particular, we have vectors $\hat v_i = (v_i, 1) \in \hat\RR^{n}$
in the lifted ellipsoid $\hEout = \Eout \times \{1\} \subset \hat\RR^n$.
Set $\hat A = A \oplus (+1)$,
similar to $A_{-}$ in the proof of Proposition \ref{prop:tracebound},
but with $(\hat A)_{n+1,n+1} = +1$.
The symmetric positive definite matrix $\hat A$
induces an inner product $\llangle \cdot, \cdot \rrangle$ in $\RR^{n+1}$.
The lifted vectors $\hat v_i$ ($1 \le i \le n+1$)
form an orthogonal basis:
indeed, for $i \ne j$ we have
\[ \llangle \hat v_i , \hat v_j \rrangle =
\langle A v_i ,  v_j \rangle + 1 = 0. \]
The following result is a reformulation 
of parts of Proposition \ref{prop:tracebound}.

\begin{coro}
\label{coro:tracebound}
Assume that
$(\hat v_i)_{1 \le i \le n+1}$
is an orthogonal basis
of $\RR^{n+1}$
under $\llangle,\rrangle$.
If $\hat v_i \in \hEout$ for all $i$ then $\tr A = 1$.
If $\hat v_i \in \hEout$ for all $i \le n$
and $\hat v_{n+1} \in \hull(\hEout) \smallsetminus \hEout$
then $\tr A < 1$.
\end{coro}

\begin{proof}
As in the proof of Proposition \ref{prop:tracebound},
write $\hat v_i = (v_i,1) \in \hEout$, $v_i \in \Eout$,
$w_i = -Av_i \in \Ein$ and $\hat w_i = (w_i,1)$.
Take $A_{-} = A \oplus (-1)$ so that $A_{-}\hat v_i = -\hat w_i$.
If $v_i \in \Eout$, we have $\langle v_i, A^2 v_i \rangle = 1$
and therefore
$\llangle \hat v_i, A_{-}\hat v_i \rrangle =
- \llangle \hat v_i, \hat w_i \rrangle = 0$.
Similarly, if $v_i \in \hull(\Eout) \smallsetminus \Eout$ then
$\llangle \hat v_i, A_{-}\hat v_i \rrangle < 0$.
Thus, in the first scenario,
if $A_{-}$ is written in the basis 
$(\hat v_i)_{1 \le i \le n+1}$ of $\RR^{n+1}$
then its diagonal entries are equal to $0$
and therefore $\tr(A_{-}) = \tr(A) - 1 = 0$.
In the second scenario, all diagonal entries are nonpositive
and at least one of them is negative
and therefore $\tr(A_{-}) = \tr(A) - 1 < 0$.
\end{proof}

\begin{lemma}
\label{lemma:vstar}
Let $(v_i)_{i \le n}$ be a basis of $\RR^n$
consisting of vectors $v_i \in \Eout$.
Assume furthermore that the vectors $\hat v_i = (v_i,1) \in \RR^{n+1}$
are orthogonal with respect to $\llangle\cdot,\cdot\rrangle$.
Then there exists a unique vector $v_{\star} \in \RR^n$
such that, for all $i \le n$,
$\llangle \hat v_{\star}, \hat v_i \rrangle = 0$
(where $\hat v_{\star} = (v_{\star},1)$).
Furthermore, $v_{\star} \ne 0$.
\end{lemma}

\begin{proof}
The family $(\hat v_i)_{i \le n}$ is a basis
of a subspace $X \subset \RR^{n+1}$ of codimension $1$.
Thus, the subspace $X^\bot \subset \RR^{n+1}$,
the orthogonal complement of $X$ under $\llangle\cdot,\cdot\rrangle$,
is a line.
Let $\tilde v$ be a generator of $X^\bot$, and write $\tilde v=(v_\star,c)$ where $v_\star\in\RR^n$ and $c\in\RR$.
If $c=0$ we have
$0 = \llangle \tilde v, \hat v_i \rrangle = \langle A\tilde v, v_i \rangle$
for all $i \le n$ and therefore $A\tilde v = 0$, a contradiction.
If $\tilde v = (0,c)$ then
$0 = \llangle (0,c), \hat v_i \rrangle = c$, a contradiction.
We may therefore assume without loss of generality that
$\tilde v = (v_\star,1)$ with $v_\star\neq 0$, completing the proof.
\end{proof}

The following result is a kind of converse of Corollary \ref{coro:tracebound}.

\begin{lemma}
\label{lemma:traceboundconverse}
Let $(v_i)_{i \le n}$ be a basis of $\RR^n$
and $v_{\star} \in \RR^n$
as in Lemma \ref{lemma:vstar}.
Then  $\tr A = |Av_{\star}|^2$.
In particular,
if $\tr(A) < 1$ then $v_{\star} \in \hull(\Eout) \smallsetminus \Eout$;
if $\tr(A) = 1$ then $v_{\star} \in \Eout$;
if $\tr(A) > 1$ then $v_{\star} \notin \Eout$.
\end{lemma}

\begin{proof}
Set $v_{n+1} = v_{\star}$.
For $i \le n+1$, set $\hat v_i = (v_i,1)$
so that $(\hat v_i)_{i \le n+1}$ is an orthogonal basis.
For $i \le n$, set $w_i = -Av_i \in \Ein$ and $t_i = 1$.
Set $t_{n+1} = 1/|Av_{n+1}|$ and $w_{n+1} = - t_{n+1} Av_{n+1}$
so that $w_{n+1} \in \Ein$ and $t_{n+1} > 0$.
For $i \le n+1$, define the hyperplane
\[ H_i = \{ q \in \RR^n \;|\; \langle q, w_i \rangle = t_i \}. \]
From orthogonality, $i \ne j$ implies $v_i \in H_j$.
Let $A_{-} = A \oplus (-1)$ be as
in the proof of Proposition \ref{prop:tracebound}.
Equation \eqref{equation:traceaminus} still holds
(with the same proof) and gives us
\[ \tr A - 1 = \sum_i \frac{ \langle w_i,  A  v_i \rangle + t_i}{(-t_i)}. \]
For $i \le n$ we have $\langle w_i,  A  v_i \rangle + t_i = -1 +1 = 0$.
For $i = n+1$ we have $\langle w_i,  A  v_i \rangle + t_i =
t_{n+1} (1-\langle Av_{n+1}, Av_{n+1}\rangle)$.
We therefore have
$\tr A - 1 = \langle Av_{n+1}, Av_{n+1}\rangle - 1$,
or, equivalently, $\tr A = |Av_{\star}|^2$.
\end{proof}

The heart of the proof of Theorem \ref{theo:euler}
is the construction of $\phi^{-1}: O(n) \to \Sfinout$
by a process similar to Gram-Schmidt 
which we call {\it adjusted orthogonalization}.
We first describe this procedure,
leaving the verification of certain technical aspects
for Lemmas \ref{lemma:monotono},
\ref{lemma:project} and \ref{lemma:vn+1}.

Given an orthogonal matrix $Q \in O(n)$ with columns $u_i$,
we obtain $v_i$ (for $i \le n$)
by performing {\it adjusted orthogonalization}
on the lifted vectors $\hat u_i = (u_i,1)$.
The procedure is illustrated in Figure~\ref{fig:8ellipses}
and is described below.
In a nutshell, we start with a family $(u_i)$ of vectors
in $\Ein \subset \RR^n$,
obtain $(\hat u_i)$ with $\hat u_i = (u_i,1) \in \RR^{n+1}$,
apply the procedure to define $(\hat v_i)$
and finally project back,
yielding the family $(v_i)$ with $\hat v_i = (v_i,1)$
and $v_i \in \Eout \subset \RR^n$.
There will be a final extra step to obtain $v_{n+1} \in \RR^n$,
in a similar but different manner.

\begin{figure}[ht]
\def\svgwidth{130mm}
\centerline{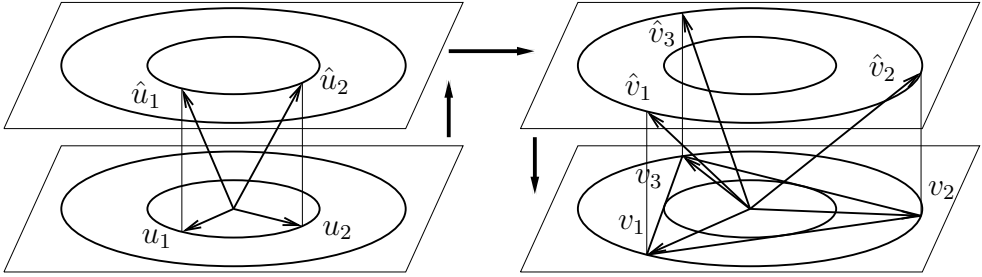}
\caption{Adjusted orthogonalization of a basis $(u_i)$
obtains vectors $(v_i)$.
In this example, $\tr A = 1$, $v_{n+1} \in \Eout$
and the simplex $S$ is tight. }
\label{fig:8ellipses}
\end{figure}


By construction, 
the family $(\hat v_j)_{1 \le j \le n}$
obtained by adjusted orthogonalization from $(\hat u_j)$ 
is orthogonal with respect to $\llangle \cdot, \cdot \rrangle$
and consists of vectors $\hat v_j \in \hEout$.
For $\hat v_j = (v_j,1)$,
the basis $(v_j)$ induces the same flag as $(u_j)$:
for any $i \le n$,
the two families $(u_j)_{j \le i}$ and $(v_j)_{j \le i}$
span the same subspace of dimension $i$ of $\RR^n$.


Set
$v_1 = u_1/ \| A u_1 \| \in \Eout$ and
$\hat v_1 = ( v_1, 1 ) \in \hEout$.
Given $1 < j < n$ and $\hat v_1, \ldots, \hat v_{j-1} \in \hEout$,
we show how to obtain $v_{j} \in \Eout \subset \RR^n$
and
$\hat v_{j} = (v_{j},1) \in \hEout$.
As in the usual Gram-Schmidt process,
we want $v_{j}$ to belong to the space spanned by
$u_1, \ldots, u_{j}$ or, equivalently, by
$v_1, \ldots, v_{j-1}, u_{j}$.
We also want the orthogonality requirement
$\llangle \hat v_{j}, \hat v_k \rrangle = 0$
(for $k < j$):
this translates into $v_{j}$ belonging to
an affine subspace of codimension $j-1$.
Thus, if we are in general position
(as we shall prove to be),
the two conditions yield that $\hat v_{j} = (v_{j},1)$
belongs to a line in $\hat\RR^{n+1}$.
As we shall prove in Lemma~\ref{lemma:vn+1},
this line intersects $\hEout$ in two distinct points.
The vector $\hat v_{j}$ is chosen to be the intersection point
for which
$v_{j} = p u_{j} - \sum_{k<j} c_k v_k$ with $p > 0$.
For reference, we itemize the crucial properties of $\hat v_{j}$:
\begin{enumerate}
\item[(0)] $\hat v_{j} = (v_{j}, 1)$,
\item[(1)] $v_{j} = p u_{j} - \sum_{k<j} c_k v_k$, $p > 0$,
\item[(2)] $\llangle \hat v_{j}, \hat v_k \rrangle = 0$, $k < j$,
\item[(3)] $\hat v_{j} \in \hEout$, or, equivalently, $\| A v_{j} \| = 1$.
\end{enumerate}

After completing the procedure for $j \le n$,
there is an extra step $j = n+1$ to obtain
the vectors $v_{n+1}$ and $\hat v_{n+1}$.
Notice first that the subspace spanned by $v_1, \ldots, v_n$ is $\RR^n$,
thus imposing no restriction on $v_{n+1}$.
Assuming general position, the orthogonality requirements
$\llangle \hat v_{j}, \hat v_k \rrangle = 0$
(for $k \le n$)
yield a single point $\hat v_{n+1}$
in the hyperplane $\hat\RR^n \subset \RR^{n+1}$.
As we shall see, if $\tr A = 1$ then $\hat v_{n+1} \in \hEout$.
On the other hand, if $\tr A < 1$ then $\hat v_{n+1}$
is in the interior of $\hEout$.

If $\tr A = 1$, the construction
obtains the vertex family $(v_j)_{j \le n+1}$ of a tight simplex $S$
associated with
the orthogonal matrix $Q$ with columns
$(u_j)_{j \le n}$.
The correspondence provides the inverse of $\phi: \Sinout \to O(n)$,
as the flags associated with the two bases
$(u_j)_{j \le n}$ and $(v_j)_{j \le n}$ coincide.

For $\tr A < 1$,
the same construction gives rise to a fitting simplex $S$
for which all hyperfaces are tangent to $\Ein$
but the last vertex $v_{n+1}$ is strictly inside $\Eout$:
the simplex is not tight.
This however suffices to derive the surjectivity of
the map $\phi: \Sinout \to O(n)$.

\medskip

We now provide details.
First, we describe adjusted orthogonalization in terms of projected vectors.
Clearly, a linearly independent  family 
$(u_j)_{j \le i}$, $u_j \in \RR^n$, lifts to a linearly independent family 
$(\hat u_j)_{j \le i}$.

For a subspace $W \subset \RR^n$,
consider the ellipsoids $\Ein \cap W$ and $\Eout \cap W$.
As above, there exists a unique real symmetric positive
linear transformation $A_W: W \to W$ with
$A_W[\Eout\cap W] = \Ein \cap W$.
If $X$ is a real symmetric matrix, we denote its spectrum by $\sigma(X)$,
which is to be interpreted as a finite multiset of real numbers.

\begin{lemma}
\label{lemma:monotono}
If $W_1 \subsetneqq W_2$ then
$\tr(A_{W_1}) < \tr(A_{W_2})$.
\end{lemma}

\begin{proof}
It suffices to consider the case $\dim(W_2) = 1+\dim(W_1) = k$.
In this case, the spectra of $A_{W_1}$ and $A_{W_2}$ are interlaced
and therefore $\tr(A_{W_1}) < \tr(A_{W_2})$, as desired.
\end{proof}


\begin{remark}
\label{remark:monotono}
The subspace $W$ is not assumed to be invariant under $A$.
Indeed, if the spectrum of $A$ is simple
there are only finitely many invariant subspaces.
In particular, $A_W$ is not to be confused with 
the restriction $A|_W: W \to A[W]$.

On the other hand, $A$ and $A_W$ define the same inner products in $W$:
if $v_1, v_2 \in W$ then
$\langle Av_1, v_2 \rangle = \langle A_W v_1, v_2 \rangle$.
The same holds for $W \oplus \RR$:
we have two expressions but the same inner product
$\llangle\cdot , \cdot\rrangle$.
\end{remark}

\begin{lemma}
\label{lemma:project}
Assume $\tr A  \le 1$ and $i \le n$.
Consider a family 
$(\hat v_j)_{j \le i}$
of vectors in $\hEout$.
If this family is orthogonal
for the inner product $\llangle \cdot, \cdot\rrangle$
then it projects to a linearly independent family 
$(v_j)_{j \le i}$.
\end{lemma}

\begin{proof}
Let $W \subseteq \RR^n$ be the subspace generated by $(v_j)_{j \le i}$.
If this family is orthogonal
for the inner product $\llangle \cdot, \cdot\rrangle$
it follows from Remark \ref{remark:monotono}
that the family is also orthogonal in $W \oplus \RR$
under the inner product $\llangle \cdot, \cdot\rrangle_W$
defined by $A_W$.
If the family $V$ is not linearly independent,
apply Corollary \ref{coro:tracebound}
to $W$, to deduce that $\tr(A_W) = 1$.
Apply Lemma \ref{lemma:monotono} 
with $W_1 = W$ and $W_2 = \RR^n$
to deduce that $\tr(A) > 1$, a contradiction.
\end{proof}

\begin{lemma}
\label{lemma:vn+1}
Assume $\tr A \le 1$ and $i < n$.
Let $X_i \subset X_{i+1} \subseteq \RR^n$
be subspaces of dimensions $i$ and $i+1$.
Let $(v_j)_{j \le i}$ be a family of vectors
in $\Eout \cap X_i$ and $\hat v_j = (v_j,1)$.
Assume that $(\hat v_j)_{j \le i}$ is orthogonal
with respect to $\llangle \cdot,\cdot \rrangle$.
Then there exist precisely two vectors $v \in \Eout \cap X_{i+1}$
such that $\hat v = (v,1)$
is $\llangle\cdot,\cdot\rrangle$-orthogonal to all $\hat v_j$, $j \le i$.
For either vector $v$, 
taking $v_{i+1} = v$ yields a basis of $X_{i+1}$,
and the two bases have opposite orientations.
\end{lemma}

\begin{proof}
From Lemma \ref{lemma:project}, $(v_j)_{j \le i}$ is a basis of $X_i$.
Apply Lemma \ref{lemma:vstar} to the restriction to $X_i$
to obtain $v_{\star} \in X_i$.
From Lemmas \ref{lemma:traceboundconverse} and \ref{lemma:monotono},
$v_\star$ is interior to $\Eout$.
Thus, the affine subspace of $X_{i+1}$ defined by the $i$ equations
$\langle \cdot, Av_i \rangle = -1$
is a line containing $v_\star$ and therefore crossing
$\Eout \cap X_{i+1}$ transversally in exactly two points.
These are our desired points.
There is one in each connected component
of $X_{i+1} \smallsetminus X_i$,
completing the proof.
\end{proof}

\begin{proof}[Proof of Theorem \ref{theo:euler}]
Recall that item (i) and the first claim in item (ii)
have already been proved in Section \ref{section:simplices}:
we just completed the proof of item (ii) by constructing
$\phi^{-1}: O(n) \to \Stinout$ when $\tr(A) = 1$.

If $\tr(A) < 1$, consider an auxiliary ellipsoid
$s\Eout$ for $s \le 1$.
The corresponding positive symmetric matrix 
is $\tilde A = s^{-1}A$ so that
$\tr(\tilde A) = s^{-1} \tr(A) \le 1$ for $s \in [1/\tr(A),1]$.
For each such value of $s$,
given $Q \in O(n)$ we obtain distinct simplices
$S_s$ with $S_s \in \Sfinout \smallsetminus \Stinout$ and $\phi(S_s) = Q$,
completing the proof of item (iii). 
\end{proof}

\section{Icosahedron and dodecahedron}

Theorems \ref{theo:fuss} and \ref{theo:euler} comprise
higher dimensional generalizations
for $k = 3$ and $k = 4$ of Poncelet porism.
It is natural to ask whether similar versions of Poncelet porism hold
for other values of $k$.
One tempting idea is to look at other regular polytopes,
such as dodecahedra and icosahedra
(recall that one is the dual of the other).
The icosahedron between two ellipsoids
does not appear to satisfy a similar porism:
we present a failed attempt.

\bigskip

We consider a special situation:
take $E_{\bout}$ to be the unit sphere $\Ss^{n-1}$
and $E_{\bin}$ to be the smaller ellipsoid with equation
\[ \frac{x^2}{a^2} + \frac{y^2}{a^2} + \frac{z^2}{b^2} = 1. \]
A convex polytope is \textit{combinatorially regular}
if it has the same combinatorial structure as the regular polytopes.
Thus, simplices, parallelotopes and cross polytopes 
are examples of combinatorially regular polytopes.
We consider convex combinatorially regular
centrally symmetric icosahedra
inscribed in $E_{\bout}$ and circumscribed to $E_{\bin}$.
We further simplify the problem by considering icosahedra
in two special positions, chosen so as to exploit symmetry.

\smallskip

We first consider icosahedra for which
the points $(0,0,\pm 1)$ are vertices.
The other $10$ vertices form two regular pentagons
on planes parallel to the $xy$ plane.
Up to rotation, the $z$ coordinate $z_{\penta} \in (0,1)$
for the top pentagon
in the only degree of freedom in the construction.
Up to symmetry, there are two faces which should be tangent to $E_{\bin}$;
it can be verified that in an open interval
$a$ and $b$ are smooth functions of $z_{\penta}$.
Applying the implicit function theorem,
for any value of $b$ between, say, $0.6$ and $0.9$,
there appears to exist a unique value of $a$
for which the construction works.
Thus, $a$ is a smooth function of $b$: $a = \alpha_1(b)$.

We next consider icosahedra with a pair of opposite faces
parallel to the $xy$ plane,
like a solid icosahedron sitting on a table.
More precisely, two opposite faces of the icosahedron
are equilateral triangles in the planes $z = \pm b$.
We may assume that the vertices of the top face are
$(s,0,b)$, $(-s/2,\pm s\sqrt{3}/2,b)$, where $s = \sqrt{1-b^2}$.
The other six vertices form two other equilateral triangles (not faces)
contained in the planes $z = \pm c$
(where $c \in (0,1)$ is another real parameter).
We may assume that the six vertices are
$\pm(-s',0,c)$, $\pm(s'/2,\pm s'\sqrt{3}/2,c)$.
Apart from the top and bottom faces,
there are, up to symmetry, two classes of faces:
the $6$ faces in the first class share an edge
with the top or bottom faces,
the $12$ faces in the other class
have three vertices with three different values of $z$.
At this point we have two degrees of freedom ($a$ and $c$)
and two equations, 
each equation verifying the tangency 
between the ellipsoid $E_{\bin}$
and the faces in one of the two classes.
Again by transversality, given $b$ there exists a unique pair $(a,c)$
for which the construction works.
In other words, $a$ is a smooth function of $b$: $a = \alpha_2(b)$.
Numerics indicate, however, that the two functions $\alpha_1$ and $\alpha_2$
are quite different.

\bigskip

\vfil

\bigskip

\noindent Peter~C.~Gibson \\
Dept.~of Mathematics \& Statistics, York University,\\
4700 Keele St., Toronto, ON, Canada, M3J~1P3 

\smallskip

\noindent Nicolau~C.~Saldanha, Carlos Tomei \\
Depto.~de~Matemática, PUC-Rio, R. Mq. S. Vicente 225,\\
Rio de Janeiro, RJ 22451-900, Brazil

\end{document}